\documentclass[10pt]{amsart}
\usepackage{amssymb}
\usepackage{epsfig}
\usepackage{url}
\usepackage[T1]{fontenc}
\theoremstyle{plain}

\newtheorem{thm}{Theorem}[section]
\newtheorem{cor}[thm]{Corollary}
\newtheorem{lem}[thm]{Lemma}
\newtheorem{prop}[thm]{Proposition}

\newtheorem{ques}[thm]{Question}
\newtheorem{conj}[thm]{Conjecture}
\newtheorem{exam}[thm]{Example}
\newtheorem*{acknow}{Acknowledgement}
\def\cal{\mathcal}
\def\bbb{\mathbb}
\def\op{\operatorname}
\renewcommand{\phi}{\varphi}

\newcommand{\N}{\bbb{N}}
\newcommand{\Z}{\bbb{Z}}
\newcommand{\Q}{\bbb{Q}}

\begin{document}

\title[Variations on twists of curves]{Variations on twists of tuples of hyperelliptic curves and related results}
\author{ Tomasz J\k{e}drzejak and Maciej Ulas}

\keywords{twist of a curve, hyperelliptic curve, superelliptic curve, rank, torsion part, Jacobian} \subjclass[2000]{11G05}
\thanks{Research of the second author was supported by Polish Government funds for science, grant IP 2011 057671 for the years 2012--2013.}

\begin{abstract}
Let $f\in\Q[x]$ be a square-free polynomial of degree $\geq 3$ and $m\geq 3$ be an odd positive integer. Based on our earlier investigations we prove that there exists a function $D_{1}\in\Q(u,v,w)$ such that the Jacobians of the curves
\begin{equation*}
C_{1}:\;D_{1}y^2=f(x),\quad C_{2}:\;y^2=D_{1}x^m+b,\quad C_{3}:\;y^2=D_{1}x^m+c,
\end{equation*}
have all positive ranks over $\Q(u,v,w)$. Similarly, we prove that there exists a function $D_{2}\in\Q(u,v,w)$ such that the Jacobians of the curves
\begin{equation*}
C_{1}:\;D_{2}y^2=h(x),\quad C_{2}:\;y^2=D_{2}x^m+b,\quad C_{3}:\;y^2=x^m+cD_{2},
\end{equation*}
have all positive ranks over $\Q(u,v,w)$. Moreover, if $f(x)=x^m+a$ for some $a\in\Z\setminus\{0\}$, we prove the existence of a function $D_{3}\in\Q(u,v,w)$ such that the Jacobians of the curves
\begin{equation*}
C_{1}:\;y^2=D_{3}x^{m}+a,\quad C_{2}:\;y^2=D_{3}x^m+b,\quad C_{3}:\;y^2=x^m+cD_{3},
\end{equation*}
have all positive ranks over $\Q(u,v,w)$. We present also some applications of these results.

Finally, we present some results concerning the torsion parts of the Jacobians of the superelliptic curves $y^p=x^{m}(x+a)$ and $y^p=x^{m}(a-x)^{k}$ for a prime $p$ and $0<m<p-2$ and $k<p$ and apply our result in order to prove the existence of a function $D\in\Q(u,v,w,t)$ such that the Jacobians of the curves \begin{equation*}
C_{1}:\;Dy^p=x^m(x+a),\quad Dy^p=x^m(x+b)
\end{equation*}
have both positive rank over $\Q(u,v,w,t)$.

\end{abstract}

\maketitle

\section{Introduction}\label{Section1}
In a recent paper \cite{JeToUl} we considered some problems related to the existence of simultaneous twists of triples and quadruplets
 of hyperelliptic curves defined by the equation $y^2=x^n+a$. More precisely, in the cited paper it is proved that for any given nonzero rational
 numbers $a, b, c$, there exists
a polynomial $d(t)\in \Q[t]$ such that the Jacobians of the curves given by $y^2 = x^n + ad(t),  y^2 = x^n + bd(t)$ and $y^2 = x^n + cd(t)$ all have
positive
rank over $\Q(t)$. In the case of odd $n$ it is possible to extend this result to four curves of the considered form. In \cite{JU} we also considered
the octic twists of the hyperelliptic curves defined by the equation $y^2=x^5+Ax$ and proved similar result.
The questions leading to this results were motivated by the work of the second author concerned with the existence of simultaneous twists
(of the same type) for tuples of elliptic curves \cite{Ul1, Ul2}.
In a recent paper \cite{Ul3} the second author considered some variations on this topic and asked about the existence of simultaneous twists
not necessarily of the same type of triples of elliptic curves.
In the cited paper, among other things it is proved that if $E_{1}:\;y^2=f(x)$ is an elliptic curve and $E_{2}, E_{3}$ are
elliptic curves with $j$-invariant 0, then there exists a rational function $D_{2,3,3}\in\Q(u,v,w)$ such that the quadratic twist of
the curve $E_{1}$ and the cubic twists of the curves $E_{2}, E_{3}$ by the $D_{2,3,3}$ have positive rank over $\Q(u,v,w)$.
 Moreover, it is proved the existence of a rational function $D_{2,6,6}\in\Q(u,v,w)$ such that the quadratic twist of the curve $E_{1}$
 and the sextic twists of the curves $E_{2}, E_{3}$ by the $D_{2,6,6}$ have positive rank over $\Q(u,v,w)$.

The aim of the present paper is to generalize the results from the paper \cite{Ul3} to the case of twists of triples of hyperelliptic curves.
We also present some results concerning twists of certain superelliptic curves.

We consider the hyperelliptic curves
\begin{equation}\label{hypcurves}
C_{1}:\;y^2=f(x),\quad C_{2}:\;y^2=x^m+b,\quad C_{3}:\;y^2=x^m+c,
\end{equation}
where $f\in\Z[x]$ is a square-free polynomial with $\op{deg}f\geq 3$, $m=2n+1\geq 3$ and $b, c\in\Z\setminus\{0\}$.
From our assumption on $f$ and $m$ we get that the genus of $C_{1}$ is at least 1 and the genus of $C_{i}$ is $n$ for $i=2,3$.
Let us recall that by the $k$-twist, with $k\geq 2$, of a curve $C$ defined over $\Q$ we understand a curve $C'$ such that $C\simeq C'$
over finite extension $K/\Q$ of degree $k$. For general hyperelliptic curve $C_{1}$,
there exists only quadratic twists, i.e. for any $d\in\Q^{*}\setminus\Q^{*2}$ we have $C_{1}':\;dy^2=f(x)$ and $C_{1}\simeq C_{1}'$ over $\Q(\sqrt{d})$.
However, in case of hyperelliptic curve defined by the equation $C:\;y^2=x^m+a$ with $m$ odd, there exists higher twists.
In particular we have $m$-twist of the curve $C$ defined by the equation $C':\;y^2=dx^m+a$.
In this case $C\simeq C'$ over $\Q(\sqrt[m]{d})$ for any given $d\in\Q$ which in its factorization into primes (with positive and negative exponents)
 contains a prime with the exponent co-prime to $m$. With $d$ chosen in this way we will say in the sequel that $C'$ is a $m$-twist of $C$ by $d$.
 Moreover for given $m$, we have $2m$-twist of the curve $C$ defined by the equation $C':\;y^2=x^m+ad$.
 In this case $C\simeq C'$ over $\Q(\sqrt[2m]{d})$ for any given $d\in\Q$ which in its factorization into primes contains a prime with the
 exponent co-prime to $2m$. With $d$ chosen in this way we will say in the sequel that $C'$ is a $2m$-twist of $C$ by $d$.
 The mentioned twists are natural generalizations of the cubic twists (in the case of $m$-twist) and the  sextic twists (in the case of $2m$-twist) of an
  elliptic curve with $j$-invariant 0. It is also clear that in case of the curve $C$ with composite $m$ it is possible to define other twists.
  However, in this paper we will concentrate only on $m$-twists and $2m$-twists which always exist.
   More information concerning twists of curves can be found in \cite{MT}.

Let us present the content of the paper in some details. In section \ref{Section2} we prove that there exists a rational function $D\in\Q(u,v,w)$
such that the Jacobians of the simultaneous quadratic twist of the curve $C_{1}$ and the $m$-twists of the curves $C_{2}, C_{3}$ by $D$
have positive rank over $\Q(u,v,w)$. Using this result we prove that there exists a $D'\in\Q(u,v,w)$ such that the Jacobian of the quadratic twist
of $C_{1}$ by $D'$ has positive rank over $\Q(u,v,w)$ and the Jacobian of $m$-twist of $C_{2}$ has rank $\geq 2$ over $\Q(u,v,w)$.
Similar result is proved in section \ref{Section3} in case of Jacobians of simultaneous quadratic twist of $C_{1}$ and the $2m$-twists of $C_{2}, C_{3}$.
In particular, we prove that there is a function $D'\in\Q(u,v,w)$ such that the Jacobian of quadratic twist of the curve $C_{1}$ by $D'$ has positive
rank over $\Q(u,v,w)$ and the Jacobian of $2m$-twists of the curve $C_{2}$ has rank $\geq 2$ over $\Q(u,v,w)$.

In section \ref{Section4} we put $f(x)=x^m+a$ with $a\in\Q\setminus\{0\}$ and consider the curve $C_{1}:\;y^2=x^m+a$.
We prove that there exists a rational function $D\in\Q(u,v,w)$ such that the Jacobians of the simultaneous $m$-twists of the curves $C_{1}, C_{2}$
and the $2m$-twist of the curve $C_{3}$ by $D$ have positive rank over $\Q(u,v,w)$. Section \ref{Section5} is devoted to the study of
superelliptic curves defined by the equation $C_{m,a}:\;y^p=x^m(x+a)$, where $p$ is a prime number and $0<m<p$.
We observe that for any given $d\in\Q$ which is not a $p$-th power we can consider $p$-twist of $C_{m,a}$ given by the equation
 $C_{m,a,d}:\;dy^p=x^m(x+a)$. First we give precise information concerning the torsion part of the Jacobian of the curve $C_{m,a}$.
 This information is used in the proof of the existence of the rational function $D\in\Q(u,v,w,t)$ such that the Jacobians of $p$-twists
 of $C_{m,a}$ and $C_{m,b}$ by $D$ have positive rank over $\Q(u,v,w,t)$.

\section{Quadratic twist and two $m$-twists}\label{Section2}

The aim of this section is to prove the following result.

\begin{thm}\label{twist2mm}
Let $f\in\Q[x]$ be given and let us suppose that $f$ has no multiple roots. Let $m, b, c\in\Z\setminus\{0\}$, where $m$ is an odd positive integer and consider the hyperelliptic curves given by {\rm(\ref{hypcurves})}. Then there exists a rational function
$D_{2,m,m}\in\Q(u,v,w)$ such that the Jacobian of quadratic twist of the curve
$C_{1}$ and the Jacobians of $m$-twists of the curves $C_{2}, C_{3}$ by $D_{2,m,m}(u,v,w)$ have positive rank over the field $\Q(u,v,w)$.
\end{thm}
\begin{proof}
We use similar idea as in the proof of Theorem 2.1 from \cite{Ul3}. Let $m=2n+1$ with $n\in\N$. In order to find the function we are looking for we need to find solutions of the following system of equations
\begin{equation}\label{2mmsys}
\frac{f(x_{1})}{y_{1}^2}=\frac{y_{2}^2-b}{x_{2}^m}=\frac{y_{3}^2-c}{x_{3}^m}.
\end{equation}
It is clear that we need solutions $x_{i}, y_{i}$ of the system (\ref{2mmsys}) which satisfy the condition
$x_{i}y_{i}f(x_{1})(y_{2}^2-b)(y_{3}^2-c)\neq 0$ for any $i=1,2,3$.

In order to find solutions of the system (\ref{2mmsys}) we make the
following substitutions
\begin{equation}\label{sub1}
  x_{1}=u,\quad  x_{2}=\frac{1}{v^2T},\quad x_{3}=\frac{1}{T^{n}}
  \quad   y_{1}=\frac{1}{T}, \quad y_{2}=p, \quad y_{3}=q,
\end{equation}
where $u,v$ are rational parameters and $p,q,T$ have to be
determined. After some simple manipulations we observe that the system (\ref{2mmsys}) simplifies to
the system
\begin{equation}\label{sys2mms2}
T=\frac{f(u)}{v^{2m}(p^2-b)},\quad v^{2m}(p^2-b)=q^2-c,
\end{equation}
which can be easily solved. Indeed, it is enough to solve only the second equation which represents the curve, say $C$ defined over the field $\Q(v)$, of genus 0 in the plane $(p,q)$ with known $\Q(v)$-rational point at infinity $[p:q:r]=[1:v^m:0]$. The parametrization of $C$ takes the form
\begin{equation}\label{pqsol2mm}
p=\frac{(b+w^2)v^{2m}-c}{2wv^{2m}},\quad q=\frac{(b-w^2)v^{2m}-c}{2wv^{m}},
\end{equation}
where $w$ is a rational parameter. Using the computed value of $p$
we get the expression for $T$ in the following form
\begin{equation}\label{Tsol233}
T=T(u,v,w)=\frac{4w^2v^{2m}f(u)}{v^{4m}w^{4}-2v^{2m}(bv^{2m}+c)w^{2}+(bv^{2m}-c)^2}.
\end{equation}
The value of $D_{2,m,m}$ we are looking for is just the common value of the expressions from the system (\ref{2mmsys}) and takes the form
\begin{equation*}
D_{2,m,m}(u,v,w)=f(u)T(u,v,w)^{m-1}.
\end{equation*}
We observe that the values of $p,q,T$ we have computed allow us to define the point
\begin{equation*}
P_{1}=\left(u, \frac{1}{T(u,v,w)^{n}}\right)
\end{equation*}
which lies on the curve $C_{1}':\;D_{2,m,m}(u,v,w)y^2=f(x)$. The curve $C_{1}'$ is the quadratic twist of the curve $C_{1}$ by
$D_{2,m,m}(u,v,w)$. Moreover, the points
\begin{align*}
&P_{2}=\left(v^{-2}f(u)T(u,v,w)^{m-2},\;f(u)^{n}pT(u,v,w)^{n(m-1)}\right),\\
&P_{3}=\left(f(u)T(u,v,w)^{m-2},\;f(u)^{n}qT(u,v,w)^{n(m-1)}\right),
\end{align*}
where $p,q,$ are given by (\ref{pqsol2mm}), lie on the curves
\begin{equation*}
C_{2}':\;y^2=x^m+bD_{2,m,m}(u,v,w)^{m-1},\quad
C_{3}':\;y^2=x^m+cD_{2,m,m}(u,v,w)^{m-1}.
\end{equation*}
The curve $C_{i}'$ is $m$-twist of the curves $C_{i}$ by $D_{2,m,m}$ respectively.

Let $J_{i}=Jac(C_{i}')$ be a Jacobian of the curve $C_{i}'$. Then the existence of $\Q(u,v,w)$-rational point $P_{i}$ on the curve $C_{i}'$ implies the existence of $\Q(u,v,w)$-rational divisor $D_{i}=(P_{i})-(\infty)$ in the Jacobian $J_{i}$ for $i=1,2,3$. It is clear that $D_{1}$ has infinite order in $J_{1}(\Q(u,v,w))$ which follows from the fact that $C_{i}'$ is a non-constant quadratic twists of a constant curve $C_{1}$.

In order to prove that $D_{i}$ is of infinite order in $J_{i}$ for $i=2,3$ we invoke the result obtained in \cite[Proposition 2.1]{JeToUl} which says that if $P=(x(t),y(t))$ is a non-constant point on the hyperelliptic curve $H:\;y^2=x^m+h(t)$ with $h\in\Q[t]\setminus\Q$ and $y\neq 0$ then the corresponding divisor $(P)-(\infty)$ lying in $Jac(H)$ is of infinite order. Applying this result in our situation we immediately get that $D_{i}$ is of infinite order in $J_{i}(\Q(u,v,w))$ for $i=2, 3$ and our result follows.
\end{proof}

From the above theorem we get the following interesting result.

\begin{cor}\label{2mmcor2}
Let $m, b\in\Z$ and consider the hyperelliptic curves
$C_{1}:\;y^2=f(x),\;C_{2}:\;y^2=x^m+b$. Then there exists a
rational function $D_{2,m}\in\Q(u,v,w)$ such that
the Jacobian of the quadratic twist of $C_{1}$ by $D_{2,m}(u,v,w)$ has positive rank
and the Jacobian of the $m$-twist of $C_{2}$ by $D_{2,m}(u,v,w)$ has rank $\geq 2$
over $\Q(u,v,w)$.
\end{cor}
\begin{proof}
First we take $b=c$ in all expressions from the proof of Theorem \ref{twist2mm}. This implies that the points $P_{2}, P_{3}$ constructed in the proof lie on the same curve $\cal{C}_{2}:=C_{2}'=C_{3}':\;y^{2}=x^{m}+bD_{2,m}^{m-1}$, where the function $D_{2,m}$ is just the function $D_{2,m,m}$ from the proof of Theorem \ref{twist2mm} with $c$ replaced by $d$. Next, we define the automorphism $\phi$ of the field $\mathbb{Q}%
(\zeta_{2m})(u,v,w)$ as follows
$$\phi(F(u,v,w))=F(u,\zeta_{2m}v,-w).$$
Here $\zeta_{2m}$ is the $2m$-th primitive root of unity. By the formulas (\ref{pqsol2mm}) and (\ref{Tsol233}) we get $\phi(p)=-p$, $\phi(q)=q$ and $\phi(T)=T$. Hence
$\phi(D_{2,m})=D_{2,m}$, so $\phi$ induces the map $\bar{\phi}$ on the
curve $\cal{C}_{2}$ and on its Jacobian, say $\cal{J}_{2}$. Note that $\bar{\phi}(P_{2})=(\zeta_{m}^{-1}x_{2},-y_{2})$ and $\bar{\phi}(P_{3})=(x_{3},y_{3})$, where $P_{2}=(x_{2},y_{2})$ and $P_{3}=(x_{3},y_{3})$ lie on the curve
$\cal{C}_{2}$ which is the $m$-twist of the curve $C_{2}.$ Consider the
$\Q(u,v,w)$-rational divisors on $\cal{C}_{2}$
\begin{align*}
D_{2}  & :=(P_{2})-(\infty), D_{3}:=(P_{3})-(\infty),\\
D_{2}^{\prime}  & :=\sum_{i=0}^{m-1}((\zeta_{m}^{-1}x_{2},-y_{2}))-m(\infty)=\sum_{i=0}^{m-1}\bar{\phi}^{i}(P_{2})-m(\infty)=\sum_{i=0}^{m-1}\bar{\phi}^{i}(D_{2}).
\end{align*}
From the proof of the Theorem \ref{twist2mm} we know that $D_{2}$ and $D_{3}$ are of infinite order in $\cal{J}_{2}(Q(u,v,w))$. Moreover, $\bar{\phi}(D_{2})=\bar{\phi}(P_{2})-\infty$, $\bar{\phi}(D_{3})=D_{3}$, and consequently $\bar{\phi}(D_{2}')=D_{2}'+(x_{2},-y_{2})-(x_{2},y_{2})\sim D_{2}'-2D_{2}$. Therefore $(1-\phi)D_{2}'=2D_{2}$, so $D_{2}'$ is of infinite order in $\cal{J}_{2}(\Q(u,v,w))$ too. It
remains to prove that the divisors $D_{2}'$ and $D_{3}$ are linearly
independent. Suppose that $\alpha D_{2}'+\beta D_{3}\sim 0$ for some
integer $\alpha$ and $\beta$. Applying the automorphism $\bar{\phi}$ we obtain
$\alpha D_{2}'-2\alpha D_{2}+\beta D_{3}\sim0$. Subtracting the first
equation from the second we have $2\alpha D_{2}\sim 0$, hence $\alpha=0$.
Consequently $\beta=0$ and we are done.
\end{proof}

\section{Quadratic twist and two $2m$-twists}\label{Section3}

The aim of this section is to prove the following result.

\begin{thm}\label{twist22m2m}
Let $f\in\Q[x]$ and suppose that $f$ has no multiple roots. Let $m$ be an odd positive integer and $b,c\in\Z$.
Consider the hyperelliptic curves given by {\rm (\ref{hypcurves})}. Then there exists a rational function
$D_{2,2m,2m}\in\Q(u,v,w)$ such that the Jacobian of the quadratic twist of the curve
$C_{1}$ and the Jacobians of the $2m$-twists of the curves $C_{2}, C_{3}$ by $D_{2,2m,2m}(u,v,w)$ have positive rank over the field $\Q(u,v,w)$.
\end{thm}
\begin{proof}We use similar idea as in the proof of the Theorem 3.1 from \cite{Ul3}. Let $m=2n+1$ with $n\in\N$.
In order to find the function we are looking for we need to find solutions of the following system of equations
\begin{equation}\label{22m2msys}
\frac{f(x_{1})}{y_{1}^2}=\frac{y_{2}^2-x_{2}^{m}}{b}=\frac{y_{3}^2-x_{3}^{m}}{c}.
\end{equation}
It is clear that we need solutions $x_{i}, y_{i}$ of the system (\ref{2mmsys}) which satisfy the condition
$x_{i}y_{i}f(x_{1})(y_{2}^2-b)(y_{3}^2-c)\neq 0$ for any $i=1,2,3$.

In order to find solutions of the system (\ref{22m2msys}) we make the
following substitutions
\begin{equation*}\label{sub2}
  x_{1}=u,\quad  x_{2}=T,\quad x_{3}=v^2T
  \quad   y_{1}=\frac{1}{T^{n}}, \quad y_{2}=pT^{n}, \quad y_{3}=qT^{n},
\end{equation*}
where $u,v$ are rational parameters and $p,q,T$ have to be
determined. After this substitution the system (\ref{22m2msys}) simplifies and we get
\begin{equation}\label{sys22m2ms1}
f(u)=\frac{p^2-T}{b}=\frac{q^2-v^{2m}T}{c}.
\end{equation}
Solving now the first and the second equation from the above system
with respect to $T$ we get that
\begin{equation*}
T=p^{2}-bf(u)=\frac{cp^2-bq^2}{c-bv^{2m}}.
\end{equation*}
In order to get solutions of the system
(\ref{sys22m2ms1}) we need to solve the equation
\begin{equation*}
v^{2m}p^2-q^2+(c-bv^{2m})f(u)=0,
\end{equation*}
which define the curve, say $C$, in $(p,q)$ plane, of genus 0 with $\Q(u,v)$-rational point
at infinity $[p:q:r]=[1:v^m:0]$. The parametrization of this curve is given by
\begin{equation}\label{pqsol22m2m}
p=\frac{v^{2m}w^2-f(u)(c-bv^{2m})}{2v^{2m}w},\quad
q=\frac{-v^{2m}w^2-f(u)(c-bv^{2m})}{2v^{m}w}.
\end{equation}
Using the computed value of $p$ we get the expression for $T$ in the following form
\begin{equation*}
T=T(u,v,w)=\frac{v^{4m}w^{4}-2f(u)v^{2m}(c+bv^{2m})w^2+f(u)^2(c-bv^{2m})^2}{4v^{4m}w^2}.
\end{equation*}
The value of $D_{2,2m,2m}$ we are looking for is just the common value of the expressions from the system (\ref{22m2msys}) and takes the form
\begin{equation*}
D_{2,2m,2m}(u,v,w)=f(u)T(u,v,w)^{m-1}.
\end{equation*}

The performed construction of $p, q$ and $T$ guarantees that the (non-constant) point
\begin{equation*}
P_{1}=\left(u,\;\frac{1}{T(u,v,w)^{n}}\right)
\end{equation*}
lies on the curve $C_{1}':\;D_{2,2m,2m}(u,v,w)y^2=f(x)$ which is the quadratic twist of the curve $C_{1}$ by
$D_{2,2m,2m}(u,v,w)$. Moreover, the non-constant points
\begin{align*}
&P_{2}=\left(T(u,v,w),\;pT(u,v,w)^{n}\right),\\
&P_{3}=\left(v^2T(u,v,w),\;qT(u,v,w)^{n}\right),
\end{align*}
with non-zero second coordinates lie on the curves
\begin{equation*}
C_{2}'':\;y^2=x^m+bD_{2,2m,2m}(u,v,w),\quad
C_{3}'':\;y^2=x^m+cD_{2,2m,2m}(u,v,w)
\end{equation*}
which are $2m$-twists of the curves $C_{2}, C_{3}$ respectively.

Using the same argument as at the end of the proof of Theorem
\ref{twist2mm} we deduce that the divisor $(P_{i})-(\infty)$ is of infinite order
in the group $Jac(C_{i}')(\Q(u,v,w))$ for $i=1,2,3$.

\end{proof}

\begin{cor}\label{22m2mcor2}
Let $m, b\in\Z$ and consider the hyperelliptic curves $C_{1}:\;y^2=f(x),\;C_{2}:\;y^2=x^m+b$. Then there exists a
rational function $D_{2,2m}\in\Q(u,v,w)$ such that the Jacobian of the quadratic twist of $C_{1}$ by $D_{2,2m}(u,v,w)$ has positive rank
and the Jacobian of the $2m$-twist of $C_{2}$ by $D_{2,2m}(u,v,w)$ has rank $\geq 2$ over $\Q(u,v,w)$.
\end{cor}
\begin{proof}
The proof goes through in a similar manner to the proof of
Corollary \ref{2mmcor2}. First we take $c=b$ in all expressions from the proof of Theorem \ref{twist22m2m}. This implies that the points $P_{2}, P_{3}$ constructed in the proof lie on the same curve $\cal{C}_{2}:=C_{2}''=C_{3}'':\;y^{2}=x^{m}+bD_{2,2m}$, where the function $D_{2,2m}$ is just the function $D_{2,2m,2m}$ constructed in the proof of Theorem \ref{twist22m2m} with $c$ replaced by $b$. Now, we consider the automorphism $\psi(F(u,v,w))=F(u,\zeta
_{2m}v,w)$ of the field $\Q(\zeta_{2m})(u,v,w)$. We have $\psi(p)=p$, $\psi(q)=-q$ and $\psi(T)=T$. Hence
$\psi(D_{2,2m})=D_{2,2m}$, so $\psi$ induces the map $\bar{\psi}$ on the
curve $\cal{C}_{2}$ and on its Jacobian, say $\cal{J}_{2}$. In this case $\bar{\psi}(P_{2})=P_{2}$ and
$\bar{\psi}(P_{3})=(\zeta_{m}x_{3},-y_{3})$, where $P_{2}$ and $P_{3}$ lie on
the curve $C_{2}^{\prime\prime}$ which is the $2m$-twist of the curve
$C_{2}$. We consider the following $\Q(u,v,w)$-rational divisors on the curve $\cal{C}_{2}$
\begin{equation*}
D_{2}:=(P_{2})-(\infty), \quad\quad D_{3}:=(P_{3})-(\infty), \quad\quad D_{3}'':=\sum_{i=0}^{m-1}\bar{\psi}^{i}(D_{3}).
\end{equation*}
We have the equalities $\bar{\psi}(D_{2})=D_{2}$, $\bar{\psi}(D_{3})=\bar{\psi}(P_{3})-(\infty)$ and $\bar{\psi}(D_{3}'')=D_{3}''-2D_{3}$. Hence $D_{2}$, $D_{3}$ and $D_{3}''$ give points of infinite order in $\cal{J}_{2}(\Q(u,v,w))$. If $\alpha D_{2}+\beta D_{3}''\sim 0$,
for some integer $\alpha$ and $\beta$, then applying $\bar{\psi}$ we obtain
$\alpha=\beta=0$, and the assertion follows.
\end{proof}

\section{Two $m$-twists and a $2m$-twist}\label{Section4}

In this section we assume that the polynomial $f$ has the form $f(x)=x^m+a$, thus the curve $C_{1}$ is given by the
equation $y^2=x^m+a$. This implies that the curve admits $m$-twists. We prove the following

\begin{thm}\label{twistmm2m}
Let $a,b,c\in\Z\setminus\{0\}$ and let $m$ be an odd positive integer. Let us consider the hyperelliptic curves
\begin{equation}\label{hypcurves1}
C_{1}:\;y^2=x_{1}^m+a,\quad C_{2}:\;y^2=x^m+b,\quad C_{3}:\;y^2=x^m+c.
\end{equation}
Then there exist a rational function $D_{m,m,2m}\in\Q(u,v,w)$ such that the
Jacobians of the $m$-twists of the curves $C_{1}, C_{2}$ and the Jacobian of the $2m$-twist of
the curve $C_{3}$ by $D_{m,m,2m}(u,v,w)$ have positive rank over the field $\Q(u,v,w)$.
\end{thm}
\begin{proof}
We use similar idea as in the proofs of previous results. We put $m=2n+1$ and are looking for parametric solutions of the following system of equations
\begin{equation}\label{mm2msys}
\frac{y_{1}^2-a}{x_{1}^m}=\frac{y_{2}^2-b}{x_{3}^m}=\frac{y_{3}^2-x_{3}^m}{c}.
\end{equation}
We are interested in solutions which satisfy the condition
$x_{i}y_{i}(y_{1}^2-a)(y_{2}^2-b)(y_{3}^2-x_{3}^m)\neq 0$ for each $i=1,2,3$.

In order to find solutions of the system (\ref{mm2msys}) we put
\begin{equation*}\label{sub3}
  x_{1}=\frac{1}{T},\quad  x_{2}=\frac{1}{u^2T},\quad x_{3}=vT
  \quad   y_{1}=p, \quad y_{2}=q, \quad y_{3}=T^{n},
\end{equation*}
where $u,v$ are rational parameters and $p,q,T$ have to been
determined. The system (\ref{mm2msys}) simplifies and after simple manipulations we get
that the above system is equivalent to the following
\begin{equation}\label{sysmm2ms2}
p^2-a=u^{2m}(q^2-b),\quad T=\frac{1}{cu^{2m}(q^2-b)+v^{m}}.
\end{equation}
The second equation is just solved and the first defines the genus
zero curve, say $C$, over the field $\Q(u)$ with $\Q(u)$-rational point at infinity
$[p:q:r]=[u^3:1:0]$. Thus, the curve $C$ can be parameterized in the following form
\begin{equation}\label{pqsolmm2m}
p=\frac{w^2+a-bu^{2m}}{2w},\quad q=\frac{-w^2+a-bu^{2m}}{2u^{m}w}.
\end{equation}

Using the computed value of $q$ we find that the value of $T$ takes the form
\begin{equation*}
T=T(u,v,w)=\frac{4w^2}{cw^4-2(ac-2v^m+bcu^{2m})w^2+c(a-bu^{2m})^2}.
\end{equation*}
From the presented construction of the solutions of the system
(\ref{sysmm2ms2}) we get the value
of $D_{m,m,2m}$ which is just the common value of the expressions form (\ref{mm2msys})
\begin{equation*}
D_{m,m,2m}(u,v,w)=\frac{w^4-2(a+bu^{2m})w^2+(a-bu^{2m})^2}{4w^2}T(u,v,w)^{m}.
\end{equation*}

Using now the computed values of $p,q,T$ we get that the non-constant points
\begin{align*}
&P_{1}=((p^2-a)T(u,v,w)^{m-1},\;p(p^2-a)^{n}T(u,v,w)^{n(m-1)}),\\
&P_{2}=\left(\frac{1}{u^2}(p^2-a)T(u,v,w)^{m-1},\;(p^2-a)^{n}qT(u,v,w)^{n(m-1)}\right),
\end{align*}
with non-zero coordinates, lie on the curves
\begin{equation*}
C_{1}':\;y^2=x^m+aD_{m,m,2m}(u,v,w)^{m-1},\;C_{2}':\;y^2=x^m+bD_{m,m,2m}(u,v,w)^{m-1}
\end{equation*}
which are the $m$-twists of the curve $C_{1}, C_{2}$ respectively.
Moreover, the non-constant point
\begin{equation*}
P_{3}=\left(vT(u,v,w),\;T(u,v,w)^{n}\right),
\end{equation*}
lies on the curve $C_{3}':\;y^2=x^3+cD_{m,m,2m}(u,v,w)$ which
is the $2m$-twist of the curve $C_{3}$. Using \cite[Proposition 2.1]{JeToUl} one more time, we get that the divisor $(P_{i})-(\infty)$ is of infinite order in the Jacobian of the curve $C_{i}'$. Our theorem is proved.
\end{proof}

\begin{cor}\label{mm2mcor2}
Let $a,c\in\Z\setminus\{0\}$ and let $m$ be an odd positive integer. Let us consider the hyperelliptic curves
$C_{1}:\;y^2=x^m+a,\;C_{2}:\;y^2=x^m+b$. Then there exists a rational function $D_{m,2m}\in\Q(u,v,w)$ such that
the Jacobian of the $m$-twist of the curve $C_{1}$ by $D_{m,2m}(u,v,w)$ has rank $\geq 2$ and
the Jacobian of the $2m$-twist of the curve $C_{2}$ by $D_{m,2m}(u,v,w)$ has positive rank
over $\Q(u,v,w)$.
\end{cor}

\begin{proof}
The proof goes through in a very similar way to the proof of
Corollaries \ref{2mmcor2} and \ref{22m2mcor2}. First we take $b=a$ in all expressions from the proof of Theorem \ref{twistmm2m} and then put $c=b$. This implies that the points $P_{1}, P_{2}$ constructed in the proof lie on the same curve $\cal{C}_{1}:=C_{1}'=C_{2}':\;y^{2}=x^{m}+aD_{m,2m}^{m-1}$, where the function $D_{m,2m}$ is just the function $D_{m,m,2m}$ from the proof of Theorem \ref{twistmm2m} with $b$ replaced by $a$. We consider the automorphism $\bar{\chi }(F(u,v,w))=F(\zeta_{2m}u,v,w)$ which induces the map $\bar{\chi}$ on the curve $\cal{C}_{1}$ and on its Jacobian. We have $\bar{\chi}(P_{1})=P_{1}$ and $\bar{\chi}(P_{2})=(\zeta_{m}^{-1}%
x_{2},-y_{2})$. Just as in the proof of Corollary \ref{22m2mcor2} we obtain that the
divisors $D_{1}:=(P_{1})-(\infty)$ and $D_{2}':=\sum_{i=0}^{m-1}\bar{\chi}^{i}(P_{2})-m(\infty)$ are linearly independent and give points of
infinite order in $Jac(\cal{C}_{1})(\Q(u,v,w))$. Moreover, the divisor $(P_{3})-(\infty)$ is of infinite order in the Jacobian of the curve $C_{2}':\;y^{2}=x^{m}+bD_{m,2m}$ and the proof is complete.
\end{proof}

\section{Remarks on twists of certain superelliptic curve}\label{Section5}

Previous sections are devoted to the study of the existence of simultaneous twists with positive rank of triples of hyperelliptic curves.
In this section we change our object of interest and consider some special families of superelliptic curves.
We are mainly interested in the superelliptic curves defined by the equation
\begin{equation*}
C_{m,a}:\;y^{p}=x^{m}(x+a),
\end{equation*}
where $p$ is a prime number and $m\in\N$ and without loss of generality $0<m<p$.
We note that $C_{m,a}$ has an automorphism $(x,y)\mapsto (x,\zeta_{p}y)$, where $\zeta_{p}$ is a primitive $p$-th root of
unity. Its Jacobian $J_{m,a}$ over $K:=\Q(\zeta_{p})$ has complex multiplication by $\Z[\zeta_{p}]$. Moreover, the genus of
$C_{m,a}$ is equal to $(p-1)/2$. The curves $y^{p}=x^{m}(x+1)^{k}$ and their
Jacobians are studied by many authors (e.g. \cite{Fa, GR, C, Tz}).  These curves are the quotients of the Fermat curve $x^{p}+y^{p}=1$.
In fact, Jacobian of the Fermat curve $x^{p}+y^{p}=1$ is $\Q$-isogenous to the product of Jacobians of the curves $C_{m,1}$ for $m=1,...,p-2$
(see \cite{Fa}). Without loss of generality
(because of a birational equivalence) we may assume that $k=1$ and $1\leq
m\leq p-2$ (or vice versa). In a similar
way one can check that the Jacobian of the twisted Fermat curve $x^{p}+y^{p}=a$ is
$\Q$-isogenous to the product of $J_{m,a}$ for $m=1,...,p-2$ (c.f. \cite{DJ}).

For any given $D\in\Z\setminus\{0\}$ which is not a $p$-th power we can consider a $p$-twist of $C_{m,a}$
given by the equation $C_{m,a,D}:\;Dy^{p}=x^{m}(x+a)$. We are interested in finding values of $D\in\Z$ such that the simultaneous $p$-twists
of $C_{m,a}$ and $C_{m,b}$ by $D$ contain rational points. This is an easy problem as we will see in the proof of lemma given below.
A more interesting question is related to the computation of torsion part of the Jacobian $J_{m,a}$ associated with the curve $C_{m,a}$.
It is clear that this knowledge is useful in proving that certain divisors are of infinite order in $J_{m,a}$.
In fact, our effort devoted to the computation of torsion part $J_{m,a}$ will occupy the main part of this section.
Combining our results we prove that the rank of Jacobian varieties associated with twists of $C_{m,a}, C_{m,b}$ by constructed $D$
have both positive ranks over an appropriate rational function field.

We start with construction of a function $D$ which will be suitable to our purposes.
We prove the following result which is slightly more general then we need but can be of independent interest.

\begin{lem}\label{T5}
\begin{enumerate}
\item Let $p, m\in\N$ and $a, b\in \Z\setminus\{0\}$ are given and suppose that $m<p$.
Then the equation
\begin{equation*}\label{superell}
\frac{x_{1}^{m}(x_{1}+a)}{y_{1}^{p}}=\frac{x_{2}^{m}(x_{2}+b)}{y_{2}^{p}}
\end{equation*}
has rational parametric (homogenous) solution depending on four parameters.

\item
Let $p, m, n\in\N$ and $a, b\in \Z\setminus\{0\}$ are given and suppose that $\gcd(p,m+n)=1$.
Then the equation
\begin{equation*}
\frac{y_{1}^p-x_{1}^{m+n}}{ax_{1}^{m}}=\frac{y_{2}^p-x_{2}^{m+n}}{bx_{2}^{m}}
\end{equation*}
has rational parametric (homogenous) solution depending on four parameters.
\end{enumerate}
\end{lem}
\begin{proof}
In order to prove the first part of our lemma we put
\begin{equation}\label{SUBxy}
x_{1}=uT,\quad y_{1}=vT, \quad x_{2}=wT, \quad y_{2}=tT,
\end{equation}
where $T$ need to be determined and $u, v, w, t$ are rational parameters. After this substitution and simple manipulations we left with linear equation in variable $T$ of the form
\begin{equation*}
\frac{u^{m}(uT+a)}{v^{p}}=\frac{w^{m}(wT+b)}{t^{p}}.
\end{equation*}
We thus get
\begin{equation*}
T=T(u,v,w,t)=\frac{bv^{p}w^{m}-au^{m}t^{p})}{u^{m+1}t^{p}-v^{p}w^{m+1}}.
\end{equation*}
Summing up our reasoning we see that if
\begin{equation}\label{suitableT}
D=D(u,v,w,t)=T^{m-p}\frac{u^{m}(uT+a)}{v^{p}}
\end{equation}
then the superelliptic curves $Dy^p=x^{m}(x+a)$ and $Dy^p=x^{m}(x+b)$ contains $\Q(u,v,w,t)$-rational points
 $P_{1}=(x_{1}, y_{1}), P_{2}=(x_{2},y_{2})$ respectively, where $x_{i}, y_{i}$ are given by (\ref{SUBxy}) and the expression for $T$ is given above.

Using similar idea we prove the second part of our lemma. If $\gcd(p,m+n)=1$ then there are $\alpha,\beta \in\N_{+}$ such that $p\alpha-(m+n)\beta=1$. In order to solve the equation from the statement
of our result we put
\begin{equation}\label{subxy}
x_{1}=uT^{\alpha},\quad y_{1}=wT^{\beta}, \quad x_{2}=vT^{\alpha}, \quad y_{2}=tT^{\beta},
\end{equation}
where $T$ need to be determined and $u, v, w, t$ are rational parameters.

After this substitution and simple manipulations we left with linear equation in variable $T$ of the form
\begin{equation*}
\frac{w^{p}T-u^{m+n}}{au^{m}}=\frac{t^pT-v^{m+n}}{bv^{m}}.
\end{equation*}
We thus get
\begin{equation*}
T=T(u,v,w,t)=\frac{u^{m}v^{m}(bu^{n}-av^{n})}{bv^{m}w^{p}-au^{m}t^{p}}.
\end{equation*}
Summing up our reasoning we see that if
\begin{equation*}
D=T^{\beta n}\frac{w^{p}T-u^{m+n}}{au^{n}}
\end{equation*}
then the superelliptic curves $y^p=x^{m}(x^n+aD)$ and $y^p=x^{m}(x^n+bD)$ contains $\Q(u,v,w,t)$-rational points
 $P_{1}=(x_{1}, y_{1}), P_{2}=(x_{2},y_{2})$ respectively, where $x_{i}, y_{i}$ are given by (\ref{subxy}) and the expression for $T$ is given above.
\end{proof}

Now we attempt to compute $\Q$-torsion part of the Jacobian of
$J_{m,a}$. We start with a well-known result. For convenience of the reader we outline the proof.

\begin{lem}\label{wzor}
Let C be a smooth projective curve of genus $g\geq1$ defined
over finite field $F_{q}$ and let J be its Jacobian. Set $N_{k}:=\#C(\mathbb{F}_{q^{k}})$ for $k\in N$. If $N_{k}=1+q^{k}$ for $k=1,2,...,g$
then $\#J(\mathbb{F}_{q})=1+q^{g}.$
\end{lem}

\begin{proof}
It is known (see for example \cite[Excercise A.8.11]{HS}) that
\begin{equation}\label{cq}
\#C\left(  \mathbb{F}_{q^{k}}\right)  =q^{k}+1-\left(  \alpha_{1}^{k}+...+\alpha_{2g}^{k}\right),
\end{equation}
where the polynomial $P(T)=\prod_{i=1}^{2g}(1-\alpha_{i}T) \in\Z[T]$ satisfies $P(T)=q^{g}T^{2g}P\left(\frac{1}{qT}\right)$. Then
\begin{equation*}
\#J(\mathbb{F}_{q})=P(1)=\prod_{i=1}^{2g}(1-\alpha_{i}),
\end{equation*}
therefore
\begin{equation}\label{jq}
\#J\left(  \mathbb{F}_{q}\right)  =\left(  -1\right)  ^{g}s_{g}+\sum
_{i=0}^{g-1}\left(  -1\right)  ^{i}\left(  1+q^{g-i}\right)  s_{i}\text{,}
\end{equation}
\ where $s_{i}=s_{i}\left(  \alpha_{1},...,\alpha_{2g}\right)  $ denote the
$i$-th fundamental symmetric polynomial (by definition $s_{0}:=1$). Let
$t_{i}=t_{i}\left(  \alpha_{1},...,\alpha_{2g}\right)  :=$\ $\alpha_{1}%
^{i}+...+\alpha_{2g}^{i}$ be the $i$-th Newton polynomial. Since by assumption
$t_{k}\left(  \alpha_{1},...,\alpha_{2g}\right)  =0$, for $k=1,...,g$, using
the Newton formulas
\begin{equation}\label{nf}
ks_{k}=\sum_{i=1}^{k}(-1)  ^{i-1}t_{i}s_{k-i}
\end{equation}
we get $s_{k}=0$, for $k=1,...,g$. Hence by (\ref{jq}), we are done.
\end{proof}

\begin{prop}\label{L5}
We have\ $\Z/p\Z\subset J_{m,a}(\Q)_{tors}\subset\Z/2p\Z.$
\end{prop}

\begin{proof}
First we will show that $J_{m,a}(\Q)_{tors}\subset\Z/2p\Z$. Observe that
\begin{equation}\label{cq1}
\#C_{m,a}\left(  \mathbb{F}_{l^{n}}\right)  =l^{n}+1\text{, if }p\nmid
l^{n}-1\text{.}
\end{equation}
Indeed, the map $x\longmapsto x^{p}$ is one-to-one on $\mathbb{F}_{l^{n}}$. If
$l$ is the primitive root modulo $p$ then the formula (\ref{cq1}) holds for $n=1$,
$2$, ..., $(p-1)/2$. Hence by Lemma \ref{wzor}, we obtain $\#J_{a,1,m}\left(  \mathbb{F}%
_{l}\right)  =l^{\left(  p-1\right)  /2}+1$. For sufficiently large primes $l$
(say $l>c:=c(p,m,a)$) the reduction modulo $l$ homomorphism induces an
embedding $J_{m,a}(\Q)_{tors}\hookrightarrow J_{m,a}(\mathbb{F}_{l})  $, therefore
\begin{equation}
\#J_{m,a}\left(  \Q\right)  _{tors}|\#J_{m,a}\left(  \mathbb{F}%
_{l}\right)  .
\end{equation}
Take a prime $q\nmid 2p.$ We will show that $J_{m,a}(\Q)$ has no $q$-torsion. Choose a prime $l>c$ such that $l$ is a
primitive root modulo $p$ and $l\equiv1\pmod{q}.$ Then $l^{\frac{p-1}{2}}+1\equiv2\pmod{q}$ hence $q\nmid\#J_{m,a}(\mathbb{F}_{l})  $.
Now we can deduce bounds for $2$-torsion and $p$-torsion. Taking a prime $l>c$ such that $l$ is a primitive root modulo $p$ and
$l\equiv1\pmod{4}$ we have $4\nmid\#J_{m,a}(\mathbb{F}_{l})$. Similarly, taking a prime $l>c$ such that
$l$ is a primitive root modulo $p$ and $l^{\frac{p-1}{2}}\not\equiv-1\pmod{p^2}$ we obtain that $p^{2}\nmid\#J_{m,a}(  \mathbb{F}_{l})  $.

Now observe that the divisor $D=((0,0))-(\infty)$ is rational and represents the point of order $p$ in $J_{a,1,m}(\Q)$.
Indeed, $D$ is not principal but $pD=\op{div}(x)$ and the assertion follows.
\end{proof}

\bigskip

Problem of existence of $2$-torsion in $J_{m,a}(\mathbb{Q})$ is more
complicated. Gross and Rohlich \cite{GR} showed that the Jacobian of the curve $y^{p}=x^{m}(1-x)^{k}$ has a
$\Q$-rational point of order 2 if and only if $p=7$ and $m^{3}\equiv k^{3}\equiv-(m+k)^{3}\pmod{7}$. We
give the condition which will be sufficient for non-existence of such point in the
Jacobian $J^{a,m,k}$ of the curve
\begin{equation*}
C^{a,m,k}:y^{p}=x^{m}(a-x)^{k},
\end{equation*}
where $0<m,k$ and $m+k<p$. Moreover it is enough (because of a birational equivalence)
to consider only $m=1$ and $0<k<p-1$ (c.f. \cite[p. 207]{GR}).
Note that (after obvious change of variables) $C^{a,m,1}=C_{m,a}$ and
$C^{1,m,k}$ is the Gross and Rohlich curve. First, we compute the zeta
function of the curve $C^{a,m,k}$. To this aim we introduce some notations.
Let $l$ be a prime of good reduction of the curve $C^{a,m,k}$, i.e. $l\neq p$
and $l\nmid a$. If $\frak{l}$ is a prime in $K=\Q(\zeta_{p})$ lying above $l$, let $\chi_{\frak{l}}$ be the $p$-power residue
symbol modulo $\frak{l}$, i.e. the character on
$(\Z[\zeta_{p}]/\frak{l})^{\times}$, with
values in $\mu_{p}$, given by
\begin{equation*}
\chi_{\frak{l}}\left(  \alpha\right)  =\zeta_{p}^{k}\Longleftrightarrow
\alpha^{\frac{N\frak{l}-1}{p}}\equiv\zeta_{p}^{k}\pmod{\frak{l}}.
\end{equation*}

Consider the Jacobi sum
\begin{equation*}
\tau_{m,k}(\frak{l}):=-\sum_{\alpha\in(\Z[\zeta_{p}]/\frak{l})  \setminus\{0,1\}}\chi_{\frak{l}}^{m}(\alpha)\chi_{\frak{l}}^{k}(1-\alpha).
\end{equation*}
Note that $\tau_{m,k}(\frak{l})\in\Z[\zeta_{p}]$ and has an absolute value $N(\frak{l})^{1/2}$ in
any complex embedding. Then we have the following formula for the zeta function.

\begin{lem} Let
$Z(C^{a,m,k},\mathbb{F}_{l},T)$ be the zeta function of the curve $C^{a,m,k}$
over $\mathbb{F}_{l}$. Then
\begin{equation*}
Z(C^{a,m,k},\mathbb{F}_{l},T)=\frac{P_{l}(T)}{(1-T)(1-lT)},
\end{equation*}
where
\begin{equation*}
P_{l}(T)=\prod_{\frak{l}\mid l}(1-\chi_{\frak{l}}^{m+k}(a)\tau_{m,k}(\frak{l}) T^{f})
\end{equation*}
and the product is taken over all primes in $\Z[\zeta_{p}]$ lying above $l$ and $f$ is the multiplicative order of
$l\pmod{p}$.
\end{lem}
\begin{proof} As in the proof of Lemma \ref{wzor} we have $P_{l}(T)
=\prod_{i=1}^{p-1}(1-\alpha_{i}T)  $ where $\alpha_{1}^{n}+\ldots+\alpha_{p-1}^{n}=l^{n}+1-\#C^{a,m,k}(\mathbb{F}_{l^{n}})$ and $P_{l}(T)=l^{(p-1)/2}T^{p-1}P_{l}(1/lT)$. Hence it is enough to consider only
$n=1,2,...,\frac{p-1}{2}$. Let $q=l^{n}$ for $1\leq n\leq\frac{p-1}{2}$. If
$p\nmid q-1$ then by (\ref{cq1}) we get $\#C^{a,m,k}(\mathbb{F}_{q})=1+q$. Assume now that $q\equiv1\pmod{p}$. We obtain
\begin{align*}
\#C^{a,m,k}(\mathbb{F}_{q})& =1+\#\{(x,y)\in\mathbb{F}_{q}\times\mathbb{F}_{q}:\;y^{p}=x^{m}(a-x)^{k}\text{and} x=0,a\}  \\
&  +\#\{(x,y)  \in\mathbb{F}_{q}\times\mathbb{F}_{q}:\;y^{p}=x^{m}(a-x)^{k}\text{and} x\neq0,a\}  \\
&  =1+\sum_{x\in\{0,a\}}1+\sum_{x\in \mathbb{F}_{q}\setminus\{0,a\}}(1+\sum_{\chi^{p}=1,\chi\neq 1}\chi(x^{m}(a-x)^{k}))  \\
&  =1+q+\sum_{\chi^{p}=1,\chi\neq 1}\sum_{x\in\mathbb{F}_{q}\setminus\{0,a\}}\chi^{m}(x)  \chi^{k}(a-x)  \\
&  =1+q+\sum_{\chi^{p}=1,\chi\neq 1}\chi^{m+k}(a)\sum_{x\in\mathbb{F}_{q}\setminus\{0,1\}}\chi^{m}(x)\chi^{k}(1-x)  \\
&  =1+q+\sum_{\chi^{p}=1,\chi\neq 1}\chi^{m+k}(a)  J(\chi^{m},\chi^{k}),
\end{align*}
where $J(\chi,\phi):=\sum_{x\in\mathbb{F}_{q}\setminus\{0,1\}}\chi(x)\phi(1-x)$ is the Jacobi sum. Therefore $\alpha_{1}^{n}+...+\alpha_{p-1}^{n}=0$, if
$f\nmid n$, and $\alpha_{1}^{n}+...+\alpha_{p-1}^{n}=-\sum_{\chi^{p}=1, \chi\neq 1}\chi^{m+k}(a)J(\chi^{m},\chi^{k})$, if $f\mid n$.
 Now the assertion follows from the classical results due to Davenport and Hasse \cite{DH}.
\end{proof}

\begin{lem} \label{L6}
If $\Z/2\Z\subset
J^{a,m,k}(\Q)$  then $(\Z/2\Z)^{p-1}\subset J^{a,m,k}(K)  $.
\end{lem}
\begin{proof} Repeat \emph{mutatis mutandis} the proof of Lemma 1.3 from \cite{GR}.
\end{proof}

\begin{prop} \label{p2}
If $a$ is an odd integer and $p\neq7$ then $J^{a,m,k}(\Q)$ has no point of order 2.
\end{prop}
\begin{proof}Suppose, on the contrary, that
$J^{a,m,k}(\Q)$ has a point of order 2. Let $\frak{l}$
be a prime of $K$ dividing 2. By Lemma \ref{L6}, we get $(\Z/2\Z)^{p-1}\subset J^{a,m,k}(K_{\frak{l}})$
where $K_{\frak{l}}$ is the completion of $K$ at the place $\frak{l}$. But
$K_{\frak{l}}$ is an unramified extension of $\Q_{2}$, hence by
\cite[Lemma 1.4]{GR}, $J^{a,m,k}$ must be ordinary over $\mathbb{F}_{2}$. We know that (see \cite{De}) an abelian
variety is ordinary over $\mathbb{F}_{l}$ if and only if all slopes in the
Newton polygon of the characteristic polynomial for its Frobenius endomorphism
are either 0 or 1. This polynomial is a reciprocal of $P_{l}$ - the numerator
of the zeta function. Since $\chi_{\frak{l}}(a)$ is a root of unity, the roots of the numerators of the zeta functions of the curves $C^{a,m,k}$ and $C^{1,m,k}$ have the same $l$-adic valuation, and consequently corresponding Newton polygons have the same slopes too. Therefore
$J^{a,m,k}$ is ordinary over $\mathbb{F}_{2}$ (note that $C^{a,m,k}$ has good
reduction at 2 because $a$ is odd) if and only if $J^{1,m,k}$ is ordinary
over $\mathbb{F}_{2}$. By \cite[Lemma 1.5]{GR}, the last condition is equivalent to $p=7$ and $m^{3}\equiv
k^{3}\equiv-(m+k)^{3}\pmod{7}$. This contradicts our assumption.
\end{proof}

In case of even $a$ the group $J^{a,m,k}(\Q)$ may have a point of
order 2. We start with some preliminary result.

\begin{lem} \label{L7}
 The curves $C^{a,1,k}$ are hyperelliptic for $k=1,(p-1)/2,p-2$.
\end{lem}
\begin{proof} First, we show that $C^{a,1,(p-1)/2}$ and $C^{a,1,p-2}$
are birationally equivalent. Indeed, consider the function
\begin{equation*}
F\left(  x,y\right)  =\left(  a-\frac{y^{p}}{\left(  a-x\right)  ^{\frac
{p-1}{2}}},\frac{y^{p-2}}{\left(  a-x\right)  ^{\frac{p-3}{2}}}\right).
\end{equation*}
Then $F(C^{a,1,(p-1)/2})\subset C^{a,1,p-2}$ and
its inverse has the form
\begin{equation*}
G\left(  x,y\right)  =\left(  a-\frac{y^{p}}{\left(  a-x\right)  ^{p-2}}%
,\frac{y^{\frac{p-1}{2}}}{\left(  a-x\right)  ^{\frac{p-3}{2}}}\right)  .
\end{equation*}

Now it suffices to show that $C^{a,1,1}$ and $C^{a,1,p-2}$ are hyperelliptic.
Substituting $(x,y) \longmapsto\left(  \frac{y}{2^{p}}+\frac
{a}{2},-\frac{x}{4}\right) $ into the equation $y^{p}=x(a-x)  $ we
get $y^{2}=x^{p}+4^{p-1}a^{2}$. Similarly, substituting
\begin{equation*}
\left(  x,y\right)  \longmapsto\left(  a-\left(  \frac{4a}{x}\right)
^{p}\left(  \frac{y}{2^{p}a^{\frac{p-1}{2}}}-\frac{1}{2}\right)  ,\left(
\frac{4a}{x}\right)  ^{p-1}\left(  \frac{y}{2^{p}a^{\frac{p-1}{2}}}-\frac
{1}{2}\right)  \right)
\end{equation*}
into the equation $y^{p}=x(a-x)^{p-2}$ we obtain $y^{2}=x^{p}+(4a)^{p-1}$, and the assertion follows.
\end{proof}

\begin{prop}\label{p3}
The group $J^{a,1,1}(\Q)$ has a point of
order 2 if and only if $a=2$. The groups $J^{a,1,(p-1)/2}(\Q)$ and $J^{a,1,p-2}(\Q)  $
have a point of order 2 if and only if $a=2^{p-2}$ (note that without loss of
generality $a$ is $p$-th power-free integer).
\end{prop}
\begin{proof} Follows from Lemma \ref{L7} and \cite[Lemma 4.3]{J}.
\end{proof}

Now we are ready to prove the main result of this section.

\begin{thm}\label{ptwists}
Let $a,b$ be odd nonzero integers. Consider the superelliptic curves
\begin{equation}\label{supcurves1}
C_{1}:\;y^{p}=x^{m}(x+a),\quad C_{2}:\;y^{p}=x^{m}(x+b),\quad
\end{equation} where $p$ is an odd prime $\neq7$ and $0<m<p-1$. Then there exists a rational function $D_{p,p}\in\Z(u,v,w,t)$ such that the
Jacobians of the $p$-twists of the curves $C_{1}, C_{2}$  by $D_{p,p}(u,v,w,t)$ have positive rank over the
field $\Q(u,v,w,t)$.
\end{thm}
\begin{proof}
Let $D_{p,p}:=D(u,v,w,t)$, where $D$ is given by the expression (\ref{suitableT}) obtained in first part of Lemma \ref{T5}. Next, from Propositions \ref{L5} and \ref{p2}, we observe that the $\Q(u,v,w,t)$-rational divisors $D_{1}=(P_{1})-(\infty)$ and $D_{2}=(P_{2})-(\infty)$, where the points $P_{1}$, $P_{2}$ have coordinates given by (\ref{SUBxy}), are of infinite order in Jacobians of $p$-twists of the curves $C_{1}$, $C_{2}$ by $D_{p,p}$, respectively. This proves the theorem.
\end{proof}

\section{Open questions and conjectures}\label{Section6}

In this section we propose some open questions and conjectures which are natural in the contents of our work.

Our first question is related to the existence of simultaneous  quadratic, $m$-twist and
$2m$-twist of the curves $C_{1}, C_{2}, C_{3}$ given by (\ref{hypcurves}) such that the Jacobians of the twisted curves have all positive rank. We were trying to prove such a result, however, without success. This
leads us to the following.

\begin{ques}\label{ques1}
Let $f\in\Q[x]$ be without multiple roots and let $m$ be an odd positive integer, let $b,c\in\Z$ and consider the hyperelliptic curves
{\rm(\ref{hypcurves})}. Is it possible to find a rational function $D\in\Q(u_{1},\ldots,u_{k})$ for some $k$ such that the Jacobian of the quadratic twist of $C_{1}$, the Jacobian of the $m$-twist of $C_{2}$ and the Jacobian of the $2m$-twist of the curve $C_{3}$
by $d$ have all positive rank over $\Q(u_{1},\ldots,u_{k})$?
\end{ques}

It is clear that this question will have positive answer provided we will be able to find rational parametric solutions
of the following system of equations
\begin{equation*}
\frac{f(x_{1})}{y_{1}^2}=\frac{y_{2}^2-b}{x_{2}^m}=\frac{y_{3}^2-x_{3}^m}{c}.
\end{equation*}

We observe that the example constructed in \cite[Example 5.2]{Ul3} can be generalized. More precisely, we present the following:

\begin{exam}
{\rm Let us take $b=c=1$. We show that the set the above system has a parametric solution for each $f\in\Q[x]\setminus\Q$ and an odd positive integer $m$.
In order to prove this we define the rational function in two variables
\begin{equation*}
D(u,v)=f(u)\left(\frac{v^2-f(u)}{2v}\right)^{2(m-1)}.
\end{equation*}
The definition of this function allows us to find the following points on corresponding curves which are quadratic, $m$-twist and $2m$-twist of the curve $C_{1}, C_{2}, C_{3}$ respectively:
\begin{equation*}
\begin{array}{ll}
  C^{1}:\;D(u,v)y^2=f(x), &   P=\left(u,\;p(u,v)^{-2m}\right), \\
  C^{2}:\;y^2=x^m+D(u,v)^{m-1}, &   Q=\left(f(u)p(u,v)^{2(m-2)},\;f(u)^{n}p(u,v)^{2n(m-1)-1}q\right), \\
  C^{3}:\;y^2=x^m+D(u,v), &    R=\left(\frac{1}{p(u,v)^2},\;p(u,v)^{2m}q(u,v)\right),
\end{array}
\end{equation*}
where
\begin{equation*}
p(u,v)=\frac{v^2-f(u)}{2v},\quad q(u,v)=\frac{v^2+f(u)}{2v}.
\end{equation*}
Using the same argument as at the end of the proof of Theorem \ref{twist2mm} we easily deduce that $(P_{i})-(\infty)$ is of infinite order in the Jacobian of $C^{i}$.
}
\end{exam}

After having obtained the example above we hoped that it is possible to get a positive answer to the Question \ref{ques1} at least in the case  when $f(x)=x^{m}+a$ and $b=c=a$. Unfortunately, we were unable to prove such a result. However, the computer experiments we performed suggest that for any given $a\in\Z\setminus\{0\}$ and an odd positive integer $m$ it is possible to find values of $D\in\Q$ which have demanding property.
We believe that it can be done in any case and this impression leads us to the following:

\begin{conj}\label{conj2}
Let $a\in\Z\setminus\{0\}$ and consider the hyperelliptic curve
$C:\;y^2=x^m+a$. Then the set of $d\in\Q$ such that
the Jacobians of the quadratic, $m$-twist and $2m$-twist of the curve $C$ by $d$ have positive rank is infinite.
\end{conj}

The last combination of three twists which we considered was three
$m$-twists. Unfortunately, also in this case we are unable to get
a general result. However, we believe that the following is true.

\begin{conj}\label{conj3}
Let $a,b,c\in\Z\setminus\{0\}$ and consider the hyperelliptic curves
{\rm(\ref{hypcurves1})}. Then the set of those $d\in\Q$ such that the Jacobian of the $m$-twist of the curve $C_{i}$ by $d$ have
positive rank for $i=1,2,3,$ is infinite.
\end{conj}

In the light of Lemma \ref{T5} and Theorem \ref{ptwists} it is natural to state the following.

\begin{ques}\label{ques4}
For which positive integer $m$ and a prime $p$ the system of equations
\begin{equation*}
\frac{x_{1}^{m}(x_{1}+a_{1})}{y_{1}^{p}}=\frac{x_{2}^{m}(x_{2}+a_{2})}{y_{2}^{p}}=\frac{x_{3}^{m}(x_{3}+a_{3})}{y_{3}^{p}}
\end{equation*}
has a rational solution $x_{i}, y_{i}$ satisfying the condition $x_{i}y_{i}(x_{i}-a_{i})\neq 0$ for $i=1,2,3$ and all $a_{1}, a_{2}, a_{3}\in\Z\setminus\{0\}$?
\end{ques}

Similarly, in the light of Propositions \ref{p2} and \ref{p3} it is natural to state the following.

\begin{ques}
\label{ques5}
Assume that $p>5$ is an odd prime, $a$ is an even (and without loss of
generality $p$-th power-free) integer, and $k\in\{2,...,p-3\}\setminus\{\frac{p-1}{2}\}$. For which such $p, a$ and $k$ the groups
$J^{a,1,k}(\Q)  $ have a point of order $2$?
\end{ques}

Note that for such $a$ and $k$ the curves $C^{a,1,k}$ are not hyperelliptic and have bad reduction at $2$. Hence the methods from the proofs of
Propositions \ref{p2} and \ref{p3} fail. Such curves with the smallest genus (equals $3$) are $C^{a,1,2}:y^{7}=x(a-x)^{2}$
and $C^{a,1,4}:y^{7}=x(a-x)^{4}$. Numerical computations suggest the following.

\begin{conj}
Assume that $p,a$ and $k$ are such as in Question \ref{ques5}. If $l$ is a prime number such that $l\equiv1\pmod{p}, l\nmid a$
and $a$ is not a $p$-th power in $\mathbb{F}_{l}^{\ast}$ then the group $J^{a,1,k}(\mathbb{F}_{l})$ has odd order. Consequently, $J^{a,1,k}(\Q)$ contains no point of order $2$.
\end{conj}

\begin{acknow}
{\rm The authors thank the referee for a careful reading of the paper, and for suggesting numerous improvements. We would like to thank Andrzej D\k{a}browski for helpful conversations.}
\end{acknow}

\bigskip

\noindent University of Szczecin, Faculty of Mathematics and Physic, Institute of Mathematics, Wielkopolska 15, 70-451 Szczecin, Poland;
email: {\tt tjedrzejak@gmail.com}
\bigskip

\noindent Jagiellonian University, Faculty of Mathematics and Computer Science, Institute of Mathematics, {\L}ojasiewicza 6, 30 - 348 Krak\'{o}w, Poland;
 email: {\tt maciej.ulas@uj.edu.pl}

 \end{document}